\documentclass[oneside,english]{amsart}
\usepackage{amsmath,amsfonts,amsthm,graphicx,pinlabel,color,comment,amssymb,tikz}
\usepackage{hyperref}

\pdfoutput=1
\usepackage{geometry}
\newtheorem{theorem}{Theorem}[section]
\newtheorem{corollary}[theorem]{Corollary}

\newtheorem{question}[theorem]{Question} 

\newtheorem*{maintheorem}{Theorem~\ref{thm:generalsliceness}}

\theoremstyle{definition}

\newtheorem{remark}[theorem]{Remark}
\newtheorem{problem}[theorem]{Problem}

\newcommand{\blue}[1]{\textcolor{blue}{#1}}
\newcommand{\red}[1]{\textcolor{red}{#1}}
\newcommand{\Z}{{\mathbb {Z}}}

\usepackage[utf8]{inputenc}

\title{Slice obstructions from genus bounds in definite 4-manifolds}
\author{Paolo Aceto, Nickolas A.\@ Castro, Maggie Miller,\\ JungHwan Park, Andr\'as Stipsicz}

\address {Université de Lille, Lille 59000, France}
\email{paoloaceto@gmail.com }

\address{Rice University, Houston TX 77005, USA}
\email{ncastro.math@gmail.com}

\address{Stanford University, Stanford CA 94301, USA}
\email{maggie.miller.math@gmail.com}

\address {Korea Advanced Institute of Science and Technology, Daejeon 34141, Republic of Korea}
\email{jungpark0817@kaist.ac.kr}

\address{R\'{e}nyi Institute of Mathematics, Budapest 1053, Hungary}
\email{stipsicz.andras@renyi.mta.hu}

\thanks{
MM is supported by a Clay Research Fellowship and a Stanford Science Fellowship. 
JP is partially supported by Samsung Science and Technology Foundation (SSTF-BA2102-02) and the POSCO TJ Park Science Fellowship. AS was partially supported by the {\it \'Elvonal} grant NKFIH KKP126683.
}

\begin{document}

\begin{abstract}
We discuss an obstruction to a knot being smoothly slice that comes from minimum-genus bounds on smoothly embedded surfaces in definite 4-manifolds. As an example, we provide an alternate proof of the fact that the (2,1)-cable of the figure eight knot is not smoothly slice, as shown by Dai--Kang--Mallick--Park--Stoffregen in 2022. The main technical input of our argument consists of gauge-theoretic obstructions to smooth small-genus surfaces representing certain homology classes in $\mathbb{CP}^2\#\mathbb{CP}^2$ proved by Bryan in the 1990s.
\end{abstract}
\maketitle

\section{Introduction}\label{sec:background}
Unless stated otherwise, all manifolds and maps in this paper are taken to be smooth. Given a closed, oriented, 4-manifold $X$ we define the 
\emph{minimal genus function} $g_X\colon H_2(X; {\mathbb {Z}})\to {\mathbb {Z}}^{\ge 0}$ 
by
\[g_X(\alpha)=\min\left\{g\mid\text{there is a smooth embedding $i\colon \Sigma_g\to X$ with $i_*\left([\Sigma_g]\right)=\alpha$}\right\},\] where $\Sigma_g$ denotes the closed, oriented surface of genus $g$. 

Although it is expected that $g_X$ encodes important information 
about the smooth topology of $X$, we know the genus function only in a handful of cases; for example, for
${\mathbb {CP}}^2$~\cite{KM94} and for the 
two $S^2$-bundles over $S^2$~\cite{Ruberman}.
In these results gauge-theoretic tools (notably the Seiberg-Witten (SW) invariants) are used.
In many other cases some partial information for $g_X$ are available, mostly resting on 
the adjunction inequality for SW invariants (see for example \cite{Li-Li:1998}).

In this paper, we prove the following main result by relating the slice genus of some knots to small values of the minimal genus function on $2\mathbb{CP}^2$. We will observe that as a corollary of this theorem, the $(2,1)$-cable of the figure eight knot is not slice.

\begin{maintheorem}
Suppose that the knot $K$ can be turned into a slice knot $($e.g., the unknot$)$ by applying whole negative twists to $K$ along disjoint disks $D_1, D_2$ 
where $D_1$ intersects $K$ algebraically once and and $D_2$ intersects $K$ algebraically three times. 
Then the $(2,1)$-cable of $K$ is not slice.
\end{maintheorem}

It is considered much harder to get precise values of
$g_X$ for manifolds with vanishing SW invariants, for example
for $2{\mathbb {CP}}^2={\mathbb {CP}}^2\# {\mathbb {CP}}^2$ -- but it is not impossible.
For example, the class $(1,3)\in H_2(2{\mathbb {CP}}^2; {\mathbb {Z}})\cong {\mathbb {Z}}\oplus {\mathbb {Z}}$ can be represented by a torus (as
the smooth cubic curve in ${\mathbb {CP}}^2$ is a torus), but
not by a 2-sphere for the following reason: 
if this class admits a spherical representative, then we can blow it up 
9 times and obtain a sphere with self-intersection 1 in 
$2{\mathbb {CP}}^2\# 9 {\overline {\mathbb {CP}}}^2$ representing the homology class $(1,3,1,\ldots, 1)$. This 2-sphere is characteristic in a simply connected 4-manifold, so its complement is spin. Blowing the 2-sphere 
down, we get a spin 4-manifold with signature $-8$, contradicting Rokhlin's famous result about the signature of a
spin 4-manifold being divisible by 16.

In a similar manner, if $(p,q)\in H_2(2{\mathbb {CP}}^2; {\mathbb {Z}})$ is represented by a 2-sphere and both $p$ and 
$q$ are odd, then we must have $p=\pm 1$
and $q=\pm 1$: blow up the 2-sphere $p^2+q^2-1$ times, and then blow it down. The parity assumption again shows that the result
is a spin manifold with $b_2^+=1$, hence by Donaldson's Theorem B we have that the resulting manifold has $b_2^-=1$, so $p^2+q^2-1=1$, implying the claim. 

It is somewhat more complicated to get further lower bounds on 
minimal genera of other classes. As an example, consider a surface $\Sigma_g\subset 2 {\mathbb {CP}}^2$ 
representing a class
of the form $(2p,2q)\in H_2(2{\mathbb {CP}}^2; {\mathbb {Z}})$
with $p,q$ odd. This homology condition ensures that the double cover of $2{\mathbb {CP}}^2$ branched along $\Sigma_g$  
is a spin manifold, hence the $\frac{10}{8}$-theorem of Furuta~\cite{Furuta:2001} can be 
applied and (as the Euler characteristic of the double involves
the genus of $\Sigma _g$, while the signature involves
$[\Sigma _g]^2$) we get a lower bound on the genus $g$.
This idea has been further developed by J. Bryan in \cite{JB},
where the ${\mathbb {Z}}/2{\mathbb {Z}}$-action on the double
branched cover has been also taken into account, resulting
in sharper lower bounds, proving in particular the following theorem.
\begin{theorem}[{\cite[Corollary~1.7]{JB}}]
\label{thm:JBresult}
The minimal genus of a surface representing the class
$(2,6)$ in $2{\mathbb {CP}}^2$ is 10.
\end{theorem}
Indeed, a genus-10 representative of the class $(2,6)$ can
be given by the connected sum of the complex surfaces
in the ${\mathbb {CP}}^2$-summands of degrees $2$ and 
$6$, respectively.
This result naturally leads to a more general question.

\begin{question}\label{2cp2question}
If $m,n>0$, is the connected sum of complex surfaces of degrees $m$ and $n$ representing the homology class $(m,n)$ in $2{\mathbb {CP}}^2$ a smoothly minimum-genus surface?
\end{question}

The answer to Question~\ref{2cp2question} is known to 
be ``yes" for some specific small $m,n$, for example ``yes" if $(m,n)=
(2,6)$ by Theorem~\ref{thm:JBresult},  and ``no" for larger values, for example whenever $n>3m$ by \cite{mmrs}. 


We rule out $m$ or $n$ being zero to avoid trivialities. For $n>2$, it is a simple exercise to check that a degree-$n$ surface in $\mathbb{CP}^2$ does not give a minimum-genus surface representing the homology class $(0,n)$ when included into $2{\mathbb {CP}}^2$, that is, $g_{2\mathbb{CP}^2}(0,n)<g_{\mathbb{CP}^2}(0)+g_{\mathbb{CP}^2}(n)$. Indeed, observe that there is an immersed 2-sphere in $\mathbb{CP}^2$ obtained by considering $n$ distinct copies of $\mathbb{CP}^1$ that represents the homology class $n$ and has $\frac{1}{2}n(n-1)$ positive self-intersections. 
We remove all but one self-intersection by increasing 
the genus of the surface, and then remove the final self-intersection by connect-summing $\mathbb{CP}^2$ to the ambient 
manifold and replacing the two local sheets of the surface 
near the remaining self-intersection with oppositely-oriented 
copies of $\mathbb{CP}^1-\mathring{D}^2$. The result is a connected surface $\Sigma$ representing the homology class $(0,n)$ in $2\mathbb{CP}^2$. Since $\Sigma$ is obtained from $n$ spheres by attaching $\frac{1}{2}n(n-1)-1$ tubes, the genus of $\Sigma$ is \[g(\Sigma)=\left(\frac{n(n-1)}{2}-1\right)-(n-1)=\frac{n^2-3n}{2}=\frac{(n-1)(n-2)}{2}-1=g_{\mathbb{CP}^2}(0)+g_{\mathbb{CP}^2}(n)-1.\]

\section{Obstructing sliceness of knots}\label{sec:slicing}
It is known that finding the minimal genus of a homology class in a 4-manifold and determining sliceness properties of knots in $S^3$ are closely related. 
After all, finding the slice genus of a knot is a 
special (relative) case of a minimal genus problem.
However, it is slightly more surprising that minimal 
genus results for closed 4-manifolds can be used to 
obstruct sliceness of knots in $S^3$ --- in the following 
we will show such an argument.

In the summer of 2022, Dai--Kang--Mallick--Park--Stoffregen \cite{dkmps} answered a long-standing open question: is the $(2,1)$-cable $E_{2,1}$ of the figure eight knot $E$ slice? In Figure~\ref{fig:fig8_surgery} we illustrate $E$ and its $(2,1)$-cable. The knot $E$ is not slice, but the connected sum $E\# E$ is slice. This might motivate one to believe that $E_{2,1}$ is slice, since it also ``looks like" two copies of $E$ glued together. Moreover, all classical knot concordance invariants fail to obstruct sliceness of $E_{2,1}$. However, a theorem of Miyazaki from the early 1990s \cite{miyazaki} implies that $E_{2,1}$ is not ribbon. Thus, $E_{2,1}$ was perhaps the simplest potential counterexample to the Slice-Ribbon conjecture. However, \cite{dkmps} answered this question in the negative, proving that $E_{2,1}$ is not slice via an obstruction arising from the involutive Heegaard Floer homology of its 
double branched cover.

In Figure~\ref{fig:fig8_surgery}, we present an alternate proof of this fact. We make use of the technical result
cited in Theorem~\ref{thm:JBresult}; this technology has been 
in the literature for multiple decades. We consider it 
surprising that the tools necessary to answer this question 
have been at hand for so long (although we also wish to 
emphasize that we are not claiming this proof is technically 
easier than the Floer-theoretic argument).

\begin{theorem}[\cite{dkmps}]\label{thm:fig8}
The $(2,1)$-cable of the figure eight knot is not smoothly slice in $B^4$.
\end{theorem}

\begin{proof}[New proof]
First we give a new argument of the standard fact that $E$ is not smoothly slice in $B^4$. Consider the left half of Figure~\ref{fig:fig8_surgery}. Starting from the top left diagram of $E$ and move left to right and then top to bottom, we illustrate an annulus $A$ in 
\[
X:=2{\mathbb {CP}}^2\smallsetminus\left(\mathring{B}^4\sqcup\mathring{B}^4\right)\cong 2{\mathbb {CP}}^2 \#\left( S^3\times I\right)
\] 
whose boundary is $(-E\times\{0\})\sqcup(U\times\{1\})$, where $U$ is the unknot. To determine the element of $H_2(X,\partial X;\mathbb{Z})$ represented by $A$, we check how many times $A$ algebraically intersects representatives of two generators. Diagrammatically, we achieve this by computing the linking number of a cross-section of $A$ with the attaching circles of 2-handles in $X$. This is visible in the middle figure of the first column of Figure~\ref{fig:fig8_surgery}. We see that $A$ represents the homology class $(1,3)$ in $H_2(X,\partial X;\mathbb{Z})\cong H_2(2\mathbb{CP}^2;\mathbb{Z})=\mathbb{Z}\oplus\mathbb{Z}$.

Of course, $U$ bounds a smooth disk in the 4-ball. If $E$ also bounded a smooth disk in the 4-ball, then we could cap off $A$ with two disks to obtain a smooth 2-sphere in $2{\mathbb {CP}}^2$ representing the homology class $(1,3)$. This contradicts the discussion in Section~\ref{sec:background}, where we argued that this homology class cannot be represented by a smooth 2-sphere. We conclude that $E$ is not smoothly slice.

Now we move onto the $(2,1)$-cable $E_{2,1}$. Consider the right half of Figure~\ref{fig:fig8_surgery}. Starting from the top left diagram of $E_{2,1}$ and moving left to right and then top to bottom, we illustrate an annulus $A'$ in $X$ whose boundary is $-(E_{2,1})\times\{0\})\sqcup(T(2,-19)\times\{1\})$. The annulus $A'$ is itself a cable of $A$; since the $0$-framing of $-E$ extends over $A$ inducing the $-[A]\cdot [A]=-10$-framing on $U$, we can cable $A$ so that one boundary of the resulting annulus $A'$ is the $(2,1)$-cable of $-E$ while the other is the $(2,1+2\cdot(-10))$-cable of $U$, i.e., the torus knot $T(2,-19)$.

As before, we find the element of $H_2(X,\partial X;\mathbb{Z})$ represented by $A'$ by computing the linking number of a cross-section of $A'$ with the attaching circles of 2-handles in $X$. This is visible in the middle figure of the third column of Figure~\ref{fig:fig8_surgery}. We see that $A'$ represents the homology class $(2,6)$ in $\mathbb{Z}\oplus\mathbb{Z}$.

The torus knot $T_{2,-19}$ bounds a smooth genus-9 surface in $B^4$, so if $E_{2,1}$ were smoothly slice then we could cap off $A'$ with a disk and a genus-9 surface to obtain a smooth genus-9 surface in $2{\mathbb {CP}}^2$ representing the homology class $(2,6)$. This contradicts 
Theorem~\ref{thm:JBresult}, so we conclude that $E_{2,1}$ is not smoothly slice.
\end{proof}

\begin{figure}
    \centering
    \labellist
    \small\hair 2pt
    \pinlabel {{\red{1}}} at 45 118
    \pinlabel {{\red{1}}} at 60 157.5
    \pinlabel {{\red{1}}} at 248.5 118
    \pinlabel {{\red{1}}} at 263.5 157.5
    \pinlabel \rotatebox{10}{{\blue{10}}} at 274.5 69.5
    \pinlabel \rotatebox{27}{{\blue{10}}} at 368 68
    \pinlabel \rotatebox{3}{{\blue{5}}} at 330 179.2
    \pinlabel \rotatebox{3}{{\blue{5}}} at 347 116.3
    \endlabellist
    \includegraphics[width=120mm]{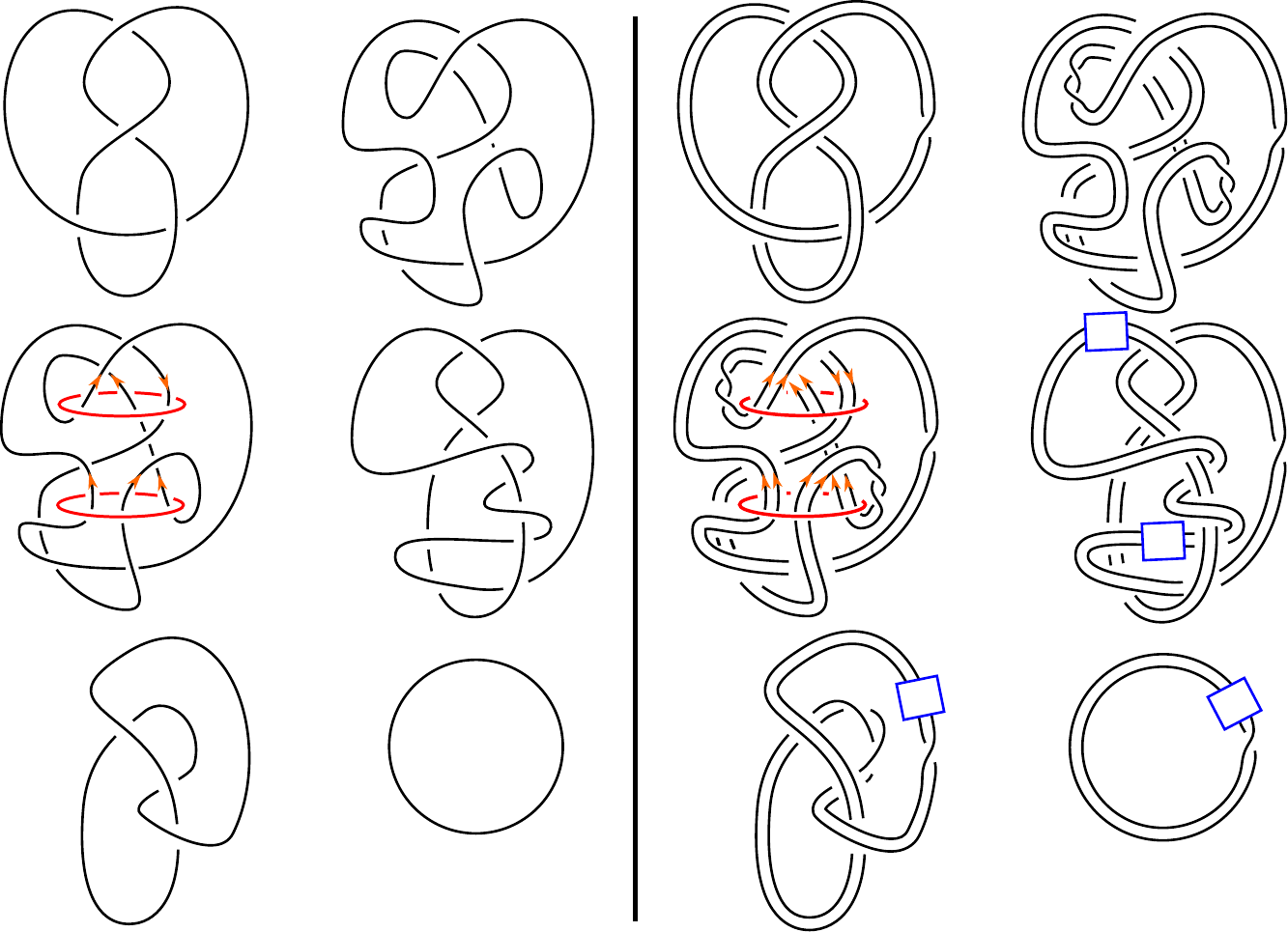}
    \caption{Illustration of the Proof of Theorem~\ref{thm:fig8}. Numbered boxes indicate numbers of whole, negative twists.}
    \label{fig:fig8_surgery}
\end{figure}
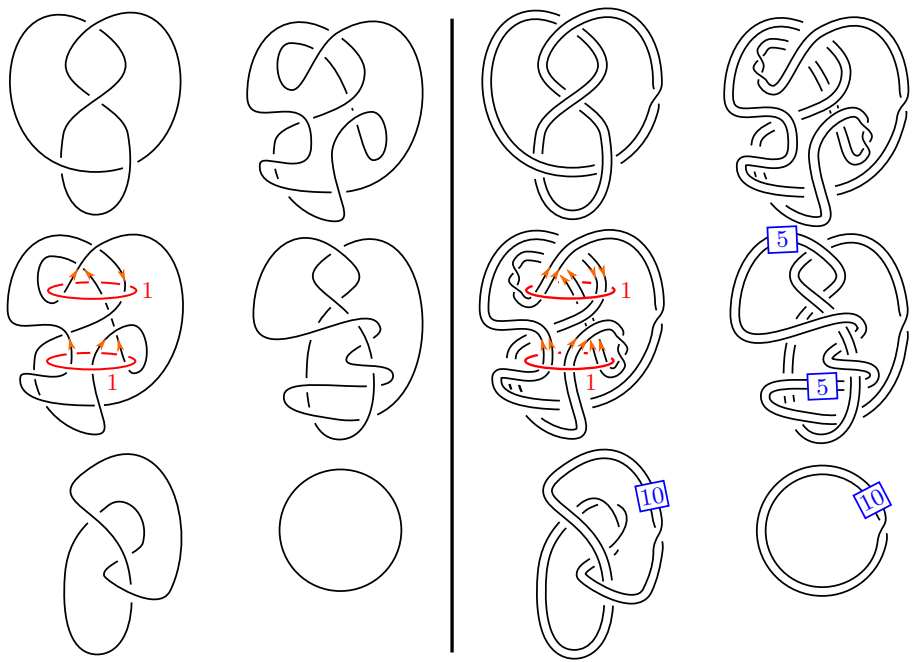

\begin{remark}
    In the proof of Theorem~\ref{thm:fig8}, we begin by constructing a disk bounded by the figure eight knot in $2\mathbb{CP}^2\smallsetminus\mathring{B}^4$, obtained by performing surgeries with the effect of removing one whole twist from one of the twist boxes in the standard diagram of the figure eight knot. This part of the construction essentially comes from Ballinger's paper on configurations of 2-spheres in $\#_n\mathbb{CP}^2$ \cite{ballinger}. Ballinger produces a 2-sphere representing the homology class $(1,3,1,3)\in H_2(4\mathbb{CP}^2;\mathbb{Z})$. His construction can be rephrased as constructing the already-mentioned disk bounded by the figure eight in $2\mathbb{CP}^2$ and another disk bounded by the figure eight knot into $2\overline{\mathbb{CP}}^2$, making use of the fact that if we {\emph{add}} (instead of remove) a twist to the figure eight knot, we can obtain the Stevedore knot, which is also slice. Gluing these two disks together along their boundary (and taking care with orientations) yields a 2-sphere in $4\mathbb{CP}^2$. The new input of Theorem~\ref{thm:fig8} is to consider cabling this disk in $2\mathbb{CP}^2\smallsetminus\mathring{B}^4$.
\end{remark}

Via the method of the proof of Theorem~\ref{thm:fig8}, we obtain the following more general result. 

\begin{theorem}\label{thm:generalsliceness}
Suppose that the knot $K$ can be turned into a slice knot $($e.g., the unknot$)$ by applying whole negative twists to $K$ along disjoint disks $D_1, D_2$ 
where $D_1$ intersects $K$ algebraically once and and $D_2$ intersects $K$ algebraically three times. 
Then the $(2,1)$-cable of $K$ is not slice.
\end{theorem}

    Theorem~\ref{thm:generalsliceness} actually obstructs sliceness in a homotopy 4-ball, since Bryan's work implying Theorem~\ref{thm:JBresult} applies to any simply-connected 4-manifold with the integral homology of $2\mathbb{CP}^2$ (see \cite{JB}).

\begin{remark}
In Theorem~\ref{thm:generalsliceness}, the knot $K$ is also not slice, but this holds more directly by observing that the Arf invariant of $K$ is $8\pmod{16}$. 
To prove this, first observe that if $K'$ is obtained from $K$ by twisting once about $n$ strands algebraically, then up to pass moves \cite{kauffman} the knot $K'$ is obtained from $K$ by twisting about $n$ strands geometrically, all oriented the same direction. This implies that twisting $K$ about $D_1$ preserves the Arf invariant \cite{kauffman}, since the resulting knot is obtainable from $K$ by pass moves. Similarly, twisting about $D_2$ changes the Arf invariant, since twisting about three parallel strands is achievable by a single sharp move \cite{ohyama,murasugiarf}. Since slice knots have vanishing Arf invariant, we conclude the Arf invariant of $K$ is nonvanishing.
\end{remark}

\begin{remark}
    More general consequences of \cite[Theorem 1.6]{JB} can easily be used to improve the statement of Theorem~\ref{thm:generalsliceness}, although we chose not to include them in the main statement as our main examples are of the form in Theorem~\ref{thm:generalsliceness}. For example, suppose that $K$ is transformed into a slice knot by adding whole negative twists about disjoint disks $D_1,\ldots, D_n$ where $n\ge 2$, $D_1$ intersects $K$ algebraically once or thrice and $D_j$ intersects $K$ algebraically thrice for all $j>1$. Applying \cite[Theorem 1.6]{JB} to the class $(2,6,\ldots, 6)$ or $(6,6,\ldots, 6)$ in $H_2(\#_n\mathbb{CP}^2;\mathbb{Z})$, we obtain a minimum genus bound  correspondingly of $9n-8$ or $9n+2$. Cabling the obvious annulus between $K$ and a slice knot yields an annulus in $\#_{n}\mathbb{CP}^2 \# \left(S^3\times I\right)$ from $K_{2,1}$ to a knot concordant to $T(2,1+2\cdot(-1-9(n-1)))=-T(2,18(n-1)+1)$ or $T(2,1+2\cdot(-9n))=-T(2,18n-1)$, which correspondingly bound surfaces of genus $9n-9$ or $9n-1$. Thus, if $D_1$ intersects $K$ algebraically once we conclude that $K_{2,1}$ is not slice; if $D_1$ intersects $K$ algebraically three times then we conclude $K_{2,1}$ has slice genus at least three.
    \end{remark}

\begin{remark}
Special situations in which Theorem~\ref{thm:generalsliceness} are easy to apply include the following.
\begin{enumerate}
\item The knot $K$ can be unknotted by adding one whole negative twist between three strands not all oriented the same direction and adding another whole negative twist between three strands that are all oriented the same direction. This is possible for the figure eight knot, as seen in the left of Figure~\ref{fig:fig8_surgery}.
\item \label{item2} The knot $K$ can be unknotted, more generally turned into a slice knot, by twice adding two whole negative twists between two strands. In the first instance, the strands must be oriented in opposite directions, and in the second instance the two strands must be oriented in the same direction. This is actually a special instance of the first situation, which the reader might deduce from Figure~\ref{fig:fig8_surgery} but we also illustrate this in Figure~\ref{fig:addtwotwists}.
\item\label{item3} In fact, there is an infinite family of such amphicheiral knots. Let $K_n$ be a 2-bridge knot which corresponds to the following presentation.
$$\begin{tikzpicture}[xscale=1.3,yscale=1,baseline={(0,0)}]
    \node at (1, .4) {$2$};
    \node at (2-0.1, .4) {$-n$};
    \node at (3, .4) {$n$};
    \node at (4-0.1, .4) {$-2$};
    \node (A1_1) at (1, 0) {$\bullet$};
    \node (A1_2) at (2, 0) {$\bullet$};
    \node (A1_3) at (3, 0) {$\bullet$};
    \node (A1_4) at (4, 0) {$\bullet$};
\path (A1_1) edge [-] node [auto] {$\scriptstyle{}$} (A1_2);
\path (A1_2) edge [-] node [auto] {$\scriptstyle{}$} (A1_3);
\path (A1_3) edge [-] node [auto] {$\scriptstyle{}$} (A1_4);
  \end{tikzpicture}$$
First, add two whole negative twist between two strands that correspond to the first vertex to get the following knot. Then we perform an isotopy to get  the weighted graph on the right.
$$\begin{tikzpicture}[xscale=1.3,yscale=1,baseline={(0,0)}]
    \node at (1-0.1, .4) {$-2$};
    \node at (2-0.1, .4) {$-n$};
    \node at (3, .4) {$n$};
    \node at (4-0.1, .4) {$-2$};
    \node (A1_1) at (1, 0) {$\bullet$};
    \node (A1_2) at (2, 0) {$\bullet$};
    \node (A1_3) at (3, 0) {$\bullet$};
    \node (A1_4) at (4, 0) {$\bullet$};
\path (A1_1) edge [-] node [auto] {$\scriptstyle{}$} (A1_2);
\path (A1_2) edge [-] node [auto] {$\scriptstyle{}$} (A1_3);
\path (A1_3) edge [-] node [auto] {$\scriptstyle{}$} (A1_4);
  \end{tikzpicture}
  \qquad \cong \qquad
  \begin{tikzpicture}[xscale=1.3,yscale=1,baseline={(0,0)}]
    \node at (1, .4) {$2$};
    \node at (2-0.1, .4) {$-n+1$};
    \node at (3, .4) {$n$};
    \node at (4-0.1, .4) {$-2$};
    \node (A1_1) at (1, 0) {$\bullet$};
    \node (A1_2) at (2, 0) {$\bullet$};
    \node (A1_3) at (3, 0) {$\bullet$};
    \node (A1_4) at (4, 0) {$\bullet$};
\path (A1_1) edge [-] node [auto] {$\scriptstyle{}$} (A1_2);
\path (A1_2) edge [-] node [auto] {$\scriptstyle{}$} (A1_3);
\path (A1_3) edge [-] node [auto] {$\scriptstyle{}$} (A1_4);
  \end{tikzpicture}
  $$
  We iterate this process one more time to get the following knot $K'_n$.
$$\begin{tikzpicture}[xscale=1.3,yscale=1,baseline={(0,0)}]
    \node at (1, .4) {$2$};
    \node at (2-0.1, .4) {$-n+2$};
    \node at (3, .4) {$n$};
    \node at (4-0.1, .4) {$-2$};
    \node (A1_1) at (1, 0) {$\bullet$};
    \node (A1_2) at (2, 0) {$\bullet$};
    \node (A1_3) at (3, 0) {$\bullet$};
    \node (A1_4) at (4, 0) {$\bullet$};
\path (A1_1) edge [-] node [auto] {$\scriptstyle{}$} (A1_2);
\path (A1_2) edge [-] node [auto] {$\scriptstyle{}$} (A1_3);
\path (A1_3) edge [-] node [auto] {$\scriptstyle{}$} (A1_4);
  \end{tikzpicture}
  $$ It is an easy exercise to verify that $K'_n$ is a slice knot for each integer $n$ (see e.g., \cite[Corollary~1.3]{Lisca:2007-1}), and also the fact that the two strands we twist are oriented in the same direction once and in the opposite direction once as in \eqref{item2}. Thus by Theorem~\ref{thm:generalsliceness}, we have that the $(2,1)$-cable of $K_n$ is not slice for each integer $n$. Note that $K_0$ is the figure eight knot and each $K_n$ is a strongly negative-amphicheiral knot. (Here, we are using the fact that every hyperbolic negative-amphicheiral knot is strongly negative amphicheiral~\cite[Lemma~1]{Kawauchi:1979-1}.) In particular, the $(2,1)$-cable of $K_n$ is an example of a knot which is strongly rationally slice but not slice (see e.g., \cite{Kawauchi:2009-1,Cha2007-1,Kang-Park:2022-1}). We point out that because each $K_n$ is alternating  and has Arf invariant $1$, it also follows from \cite{dkmps} that the $(2,1)$-cable of $K_n$ is not smoothly slice.
\end{enumerate}

\end{remark}

\begin{figure}
\centering
\includegraphics[width=80mm]{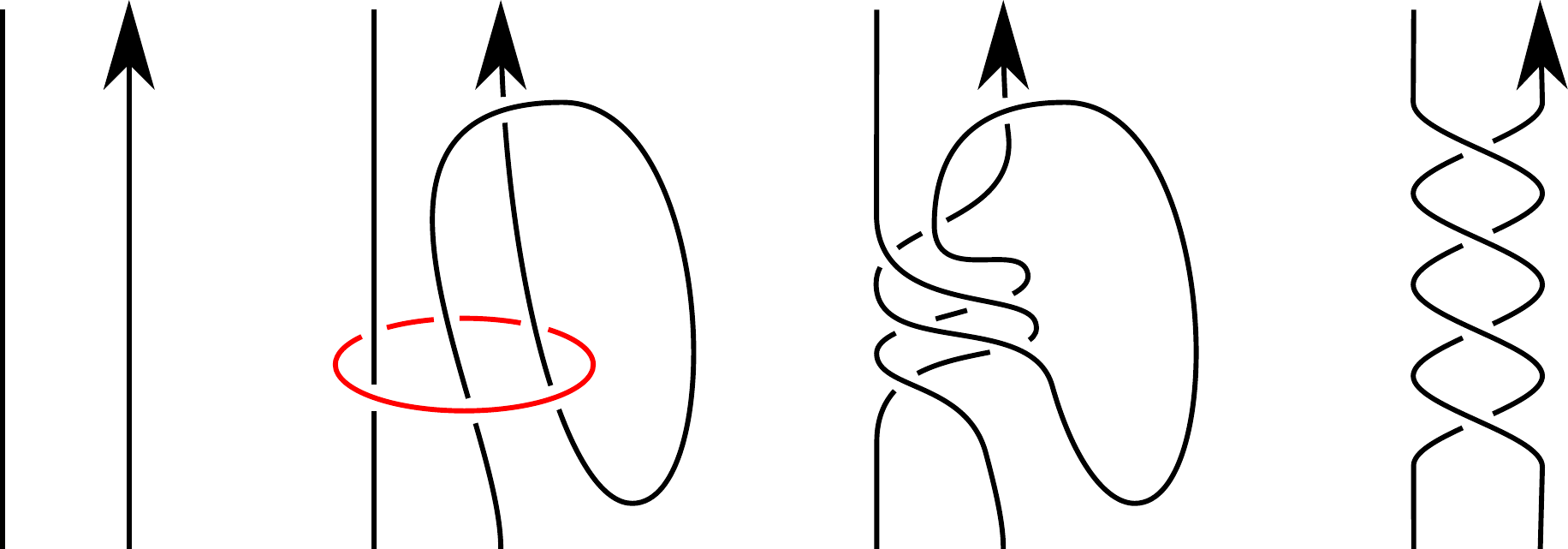}
\caption{On the left we draw two strands in a knot $K$, indicating an orientation on one strand. In the second frame, we isotope $K$ and draw the boundary of a disk intersecting $K$ geometrically in three points. Depending on the orientation of the other strand of $K$, this disk intersects $K$ algebraically once or three times (up to sign). In the third frame, we add a whole negative twist to $K$ along this disk, and then in the fourth frame isotope the resulting knot to see that it is obtained from $K$ by adding two negative whole twists between the original two pictured strands.}\label{fig:addtwotwists}
\end{figure}

The proofs of Theorems~\ref{thm:fig8} and~\ref{thm:generalsliceness} only hold in the smooth category; Theorem~\ref{thm:JBresult} is not true locally flatly since Lee and Wilczy\'{n}ski have constructed a locally flat embedding of a genus-8 surface into $2\mathbb{CP}^2$ that represents the homology class $(2,6)$ \cite{leewilczynski}. Similarly, the original proof of \cite{dkmps} only holds in the smooth category due to its use of Floer homology. We are thus left with the following still-open question.

\begin{question}
    Is the $(2,1)$-cable of the figure eight knot topologically slice?
\end{question}
 By \cite{miyazaki}, $E_{2,1}$ is not homotopy-ribbon, so if $E_{2,1}$ is topologically slice then this would give a counterexample to the topological locally flat version of the slice-ribbon conjecture (which states that topologically slice knots are homotopy-ribbon; see \cite{friedl} for further discussion). A further natural extension of our
 result would address the (2n,1)-cables.

\begin{question}\label{question:highercables}
For $k>1$, is the $(2k,1)$-cable of the figure eight knot slice?
\end{question}
Again, neither our techniques nor those of \cite{dkmps} happen to obstruct sliceness for higher cables, but at least in principle one could hope to use either set of ideas.
The construction in the presented proof of Theorem~\ref{thm:fig8} can be repeated for the $(2k,1)$ cable, yielding an annulus in $2\mathbb{CP}^2\smallsetminus\left(\mathring{B}^4\sqcup\mathring{B}^4\right)$ cobounded by $E_{2k,1}$ and the torus knot $T(2k,1-20k)$, which has 4-ball genus \[\frac{(2k-1)\cdot(20k-2)}{2}=20k^2-12k+1.\] Assuming that $E_{2k,1}$ is slice, we obtain a smooth genus-$(20k^2-12k+1)$ surface in $2\mathbb{CP}^2$ representing the homology class $(2k,6k)$. Observe that the obvious surface representing this homology class, obtained from connect-summing complex surfaces of degrees $2k$ and $6k$ in either summand, has genus 
\begin{equation}\label{eq:cpxgenus}
\frac{(2k-1)(2k-2)}{2}+\frac{(6k-1)(6k-2)}{2}=2k^2-3k+1+18k^2-9k+1=20k^2-12k+2.
\end{equation}
\begin{corollary}\label{question3}
If the answer to Question~\ref{2cp2question} is ``yes" for $m=2k,n=6k$ -- i.e., $g_{2\mathbb{CP}^2}(2k,6k)=g_{\mathbb{CP}^2}(2k)+g_{\mathbb{CP}^2}(6k)$ -- then $E_{2k,1}$ is not slice.
\end{corollary}
Note that the construction of \cite{mmrs} does not yield a small-genus surface in the homology class $(2k,6k)$, since $6k=3\cdot 2k$.


\begin{remark}
    The method of \cite{JB} applies to the 
    class $(2k,6k)\in H_2(2\mathbb{CP}^2;\mathbb{Z})$. When $k$ is odd, if $\Sigma$ represents $(2k,6k)$ then we obtain 
    $g(\Sigma)\geq 12.5 k^2-2.5$. This lower bound is much smaller than the
    genus of the connected-sum of the complex curves shown in 
    Equation~\eqref{eq:cpxgenus}. It would be interesting 
    to determine the value of the minimal genus function for 
    these homology classes. 
\end{remark}


\section{Small genus surfaces in rational homology $2\mathbb{CP}^2$}
The above results can be interpreted in terms of minimal genus functions
of closed 4-manifolds introduced at the beginning
of Section~\ref{sec:background}.
In this language, Theorem~\ref{thm:JBresult} says that
$g_{2{\mathbb {CP}}^2}(2,6)=10$, and $g_{2{\mathbb {CP}}^2} 
(1,3)=1$. 
For some further values of $g_{2{\mathbb {CP}}^2}$ 
see \cite{JB}.

It is known that the figure eight knot $E$ is \emph{rationally slice},
that is, there is a rational homology 4-ball $Z$ with boundary $S^3$ in which
$E$ bounds a smoothly embedded disk (see, e.g., \cite{Kawauchi:1979-1, Kawauchi:2009-1, Six,Levine:2022} for an explicit construction of $Z$). 
Define the closed 4-manifold $M$ as
\[
M=Z\cup _{S^3}\left(2{\mathbb {CP}}^2\smallsetminus \mathring{B}^4\right).
\]
It follows that $H_2(M; \Z )=\Z \oplus \Z$ and $H_1(M; \Z )=\Z/2\Z$.
As $E$ is slice in $Z$, our previous construction shows that
\[
g_M(1,3)=0.
\]
As the (2,1)-cable of $E$ is also slice in $Z$ (by cabling the slice disk),
our new proof of Theorem~\ref{thm:fig8} can be interpreted as
\[
g_M(2,6)\leq 9.
\]

In short, the comparison of $g_{2{\mathbb {CP}}^2}$ and of $g_M$
provides the fundamental idea of our new proof of Theorem~\ref{thm:fig8}.
It would be interesting to find further discrepancies between the values
of $g_{2{\mathbb {CP}}^2}$ and of $g_M$ on pairs of
integers $(a,b)\in \Z \oplus \Z$.

\begin{problem}\label{prob}
For each $m,n>0$ with $g_{2\mathbb{CP}^2}(m,n)>0$, construct a 4-manifold $X_{(m,n)}$ which is a rational homology $2\mathbb{CP}^2$ such that there exists a smooth surface $\Sigma$ embedded in $X_{(m,n)}$ with $[\Sigma]=(m,n)\in H_2(X_{(m,n)};\mathbb{Z})$ and $g(\Sigma)<g_{2\mathbb{CP}^2}(m,n).$
\end{problem}

A stronger version of Problem~\ref{prob} would be to find a 4-manifold $X$ which is a rational homology $2\mathbb{CP}^2$ so that for {\emph{every pair}} $m,n>0$ with $g_{2\mathbb{CP}^2}(m,n)>0$, there exists a smooth surface $\Sigma$ in $X$ with $[\Sigma]=(m,n)\in H_2(X;\mathbb{Z})$ and $g(\Sigma)<g_{2\mathbb{CP}^2}(m,n)$.

\begin{remark}
Notice that the Poincar\'e dual $c$ of the class $(1,3)\in H_2(M; \Z )$
satisfies
\[
c(\alpha )\equiv Q_M (\alpha , \alpha ),
\]
but (as $M$ has nontrivial 2-torsion in its first homology and therefore in its second cohomology) this congruence does
not imply that the mod 2 reduction of $c$ is equal to $w_2(M)\in H^2(M; \Z/2\Z)$. In particular, the
complement of the sphere in $M$ we found representing $(1,3)$ is not spin, hence
the contradiction shown in $2{\mathbb {CP}}^2$ (relying on Rokhlin's theorem)
does not apply in this context.
\end{remark}

\subsection*{Acknowledgments}The authors would like to thank the American Institute of Mathematics (AIM); this project began at the November 2022 meeting of the ``Fibered ribbon knots and Casson-Gordon exotic 4-spheres" SQuaRE (comprised of the five authors).

\bibliographystyle{alpha} 
\bibliography{biblio}

\newcommand{\etalchar}[1]{$^{#1}$}
\begin{thebibliography}{AMM{\etalchar{+}}21}

\bibitem[AMM{\etalchar{+}}21]{Six}
Paolo Aceto, Jeffrey Meier, Allison~N. Miller, Maggie Miller, JungHwan Park,
  and Andr\'{a}s~I. Stipsicz.
\newblock Branched covers bounding rational homology balls.
\newblock {\em Algebr. Geom. Topol.}, 21(7):3569--3599, 2021.

\bibitem[Bal22]{ballinger}
William Ballinger.
\newblock Configurations of spheres in $\#^n \mathbb{CP}^2$, 2022.
\newblock arXiv:2209.12851.

\bibitem[Bry98]{JB}
Jim Bryan.
\newblock Seiberg-{W}itten theory and {${\bf Z}/2^p$} actions on spin
  {$4$}-manifolds.
\newblock {\em Math. Res. Lett.}, 5(1-2):165--183, 1998.

\bibitem[Cha07]{Cha2007-1}
Jae~Choon Cha.
\newblock The structure of the rational concordance group of knots.
\newblock {\em Mem. Amer. Math. Soc.}, 189(885):x+95, 2007.

\bibitem[DKM{\etalchar{+}}22]{dkmps}
Irving Dai, Sungkyung Kang, Abhishek Mallick, JungHwan Park, and Matthew
  Stoffregen.
\newblock The $(2,1)$-cable of the figure-eight knot is not smoothly slice,
  2022.
\newblock ArXiv:2207.14187 [math.GT] July 2022.

\bibitem[FKL{\etalchar{+}}22]{friedl}
Stefan Friedl, Takahiro Kitayama, Lukas Lewark, Matthias Nagel, and Mark
  Powell.
\newblock Homotopy ribbon concordance, {B}lanchfield pairings, and twisted
  {A}lexander polynomials.
\newblock {\em Canad. J. Math.}, 74(4):1137--1176, 2022.

\bibitem[Fur01]{Furuta:2001}
Mikio Furuta.
\newblock Monopole equation and the {$\frac{11}8$}-conjecture.
\newblock {\em Math. Res. Lett.}, 8(3):279--291, 2001.

\bibitem[Kau83]{kauffman}
Louis~H. Kauffman.
\newblock {\em Formal knot theory}, volume~30 of {\em Mathematical Notes}.
\newblock Princeton University Press, Princeton, NJ, 1983.

\bibitem[Kaw79]{Kawauchi:1979-1}
Akio Kawauchi.
\newblock The invertibility problem on amphicheiral excellent knots.
\newblock {\em Proc. Japan Acad. Ser. A Math. Sci.}, 55(10):399--402, 1979.

\bibitem[Kaw09]{Kawauchi:2009-1}
Akio Kawauchi.
\newblock Rational-slice knots via strongly negative-amphicheiral knots.
\newblock {\em Commun. Math. Res.}, 25(2):177--192, 2009.

\bibitem[KM94]{KM94}
Peter~B. Kronheimer and Tomasz~S. Mrowka.
\newblock The genus of embedded surfaces in the projective plane.
\newblock {\em Math. Res. Lett.}, 1(6):797--808, 1994.

\bibitem[KP22]{Kang-Park:2022-1}
Sungkyung Kang and JungHwan Park.
\newblock Torsion in the knot concordance group and cabling, 2022.
\newblock ArXiv:2207.11870 [math.GT] July 2022.

\bibitem[Lev22]{Levine:2022}
Adam~S. Levine.
\newblock A note on rationally slice knots, 2022.
\newblock ArXiv:2212.12951 [math.GT] December 2022.

\bibitem[Lis07]{Lisca:2007-1}
Paolo Lisca.
\newblock Lens spaces, rational balls and the ribbon conjecture.
\newblock {\em Geom. Topol.}, 11:429--472, 2007.

\bibitem[LL98]{Li-Li:1998}
Bang-He Li and Tian-Jun Li.
\newblock Minimal genus smooth embeddings in {$S^2\times S^2$} and {${\bf
  C}{\rm P}^2\#n\overline{{\bf C}{\rm P}}{}^2$} with {$n\leq8$}.
\newblock {\em Topology}, 37(3):575--594, 1998.

\bibitem[LW97]{leewilczynski}
Ronnie Lee and Dariusz~M. Wilczy\'{n}ski.
\newblock Representing homology classes by locally flat surfaces of minimum
  genus.
\newblock {\em Amer. J. Math.}, 119(5):1119--1137, 1997.

\bibitem[Miy94]{miyazaki}
Katura Miyazaki.
\newblock Nonsimple, ribbon fibered knots.
\newblock {\em Trans. Amer. Math. Soc.}, 341(1):1--44, 1994.

\bibitem[MMRS22]{mmrs}
Allison Miller, Marco Marengon, Aru Ray, and Andr\'as Stipsicz.
\newblock A note on surfaces in $\mathbb {CP}^2$ and $\mathbb {CP}^2\# \mathbb
  {CP}^2$, 2022.
\newblock arXiv:2210.12486.

\bibitem[Mur69]{murasugiarf}
Kunio Murasugi.
\newblock The {A}rf invariant for knot types.
\newblock {\em Proc. Amer. Math. Soc.}, 21:69--72, 1969.

\bibitem[Ohy94]{ohyama}
Yoshiyuki Ohyama.
\newblock Twisting and unknotting operations.
\newblock {\em Rev. Mat. Univ. Complut. Madrid}, 7(2):289--305, 1994.

\bibitem[Rub96]{Ruberman}
Daniel Ruberman.
\newblock The minimal genus of an embedded surface of non-negative square in a
  rational surface.
\newblock {\em Turkish J. Math.}, 20(1):129--133, 1996.

\end{thebibliography}
\end{document}